\documentclass[leqno,11pt]{amsart}
\usepackage{amssymb, amsmath,amsmath,latexsym,amssymb,amsfonts,amsbsy, amsthm,esint}
\usepackage{color}
\usepackage{comment}
\usepackage{graphicx}
\usepackage{tikz}
\usetikzlibrary{arrows, calc, decorations.markings,intersections, patterns,patterns.meta}
\usepackage{caption}
\usepackage{subcaption}
\usepackage{hyperref}

\usepackage{pgfplots}
\pgfplotsset{compat=1.15}
\usepackage{mathrsfs}

\makeatletter
\DeclareFontFamily{U}{tipa}{}
\DeclareFontShape{U}{tipa}{m}{n}{<->tipa10}{}
\newcommand{\arc@char}{{\usefont{U}{tipa}{m}{n}\symbol{62}}}%

\newcommand{\arc}[1]{\mathpalette\arc@arc{#1}}

\newcommand{\arc@arc}[2]{%
  \sbox0{$\m@th#1#2$}%
  \vbox{
    \hbox{\resizebox{\wd0}{\height}{\arc@char}}
    \nointerlineskip
    \box0
  }%
}
\makeatother

\let\pa=\partial
\let\al=\alpha

\let\g=\gamma

\let\lam=\lambda

\let\f=\frac

\let\Om=\Omega

\let\e=\varepsilon
\let\pa=\partial

\let\ri=\rightarrow
\let\na=\nabla

\newcommand{\beq}{\begin{equation}}
\newcommand{\eeq}{\end{equation}}
\newcommand{\beqo}{\begin{equation*}}
\newcommand{\eeqo}{\end{equation*}}
\newcommand{\ben}{\begin{eqnarray}}
\newcommand{\een}{\end{eqnarray}}
\newcommand{\beno}{\begin{eqnarray*}}
\newcommand{\eeno}{\end{eqnarray*}}

\newtheorem{theorem}{Theorem}[section]
\newtheorem{definition}[theorem]{Definition}
\newtheorem{lemma}[theorem]{Lemma}
\newtheorem{proposition}[theorem]{Proposition}
\newtheorem{corol}[theorem]{Corollary}

\theoremstyle{remark}
\newtheorem{step}{Step}

\newtheorem{rmk}{Remark}[section]

\newcommand{\dist}{\mathrm{dist}}

\newcommand{\BR}{\mathbb{R}}

\newcommand{\bx}{\mathbf{x}}
\newcommand{\bn}{\mathbf{n}}
\newcommand{\bt}{\boldsymbol{\tau}}

\begin{document}
\title{Classification of minimizing solutions to a two-dimensional Allen-Cahn system}

\author{Zhiyuan Geng}
\address{Department of Mathematical Sciences, University of Arkansas, 850 West Dickson Street, Fayetteville, AR 72701}
\email{zgeng@uark.edu}

\begin{abstract}
    We study bounded entire solutions $u:\BR^2\to \BR^2$ that minimize the Allen-Cahn functional
    \begin{equation*}
        J(u,\Omega)=\int_\Omega \left(\f12 |\na u|^2+W(u)\right)\,d\bx,
    \end{equation*}
    with the $D_3$-invariant triple-well potential
    \begin{equation*}
        W(u_1,u_2)=|u|^4+2u_1u_2^2-\f23 u_1^3-|u|^2+\f23.
    \end{equation*}
    We obtain a complete classification of entire minimizing solutions. In particular, when $u$ has a triple-junction structure at infinity, up to translation and orthogonal change of coordinates, $u$ has the explicit profile
    \begin{equation*}
        u_*(\bx)=\sum_{i=1}^3 \f{e^{\sqrt2 a_i\cdot \bx}}{\sum_{j=1}^3 e^{\sqrt2 a_j\cdot \bx}}a_i.
    \end{equation*}
    We also demonstrate that the solutions obtained by minimizing within the $D_3$-equivariant class coincide with $u_*$. The key ingredient is a calibration identity arising from the special algebraic structure of $W$.
\end{abstract}

\keywords{Allen-Cahn system, heteroclinic connection, entire minimizing solutions, triple-junction, classification of solutions}

\date{\today}
\maketitle

\section{Introduction}

In this paper, we investigate solutions of the two-dimensional Allen-Cahn system
\begin{equation}
\label{EL eq} \Delta u-W_u(u)=0,  \quad u:\mathbb{R}^2\to \mathbb{R}^2,  
\end{equation}
which minimize on compact sets the associated energy functional:
\begin{equation}\label{ene func}
     J(u,\Om):=\int_\Omega \left(\f12 |\na u|^2+W(u)\right)\,d\bx,\quad \Omega\subset\mathbb{R}^2.
\end{equation}
Here we consider the specific potential
\begin{equation}
    \label{potential specific} W(u_1,u_2)=|u|^4+2u_1u_2^2-\f23 u_1^3-|u|^2+\f23. 
\end{equation}
In polar coordinates with $(u_1, u_2)=(R\cos\theta,R\sin\theta)$, the potential takes the form 
\begin{equation*}
    W(R,\theta)= R^4-R^2+\frac23-\f23 R^3\cos{(3\theta)}.
\end{equation*}
It is straightforward to verify that $W$ has the following properties (see e.g. \cite{AlikakosFusco2009,alikakos2011entire}):
\begin{enumerate}
    \item[(P1).] $W\in C^\infty(\mathbb{R}^2;\mathbb{R}^+\cup\{0\})$ and $\{\bx: W(\bx)=0\}=\{a_1,a_2,a_3\}$ where
    \begin{equation}\label{coord: ai}
        a_1:=(1,0),\quad a_2:=(-\f12,\f{\sqrt{3}}2), \quad a_3:=(-\f12,-\f{\sqrt{3}}2).
    \end{equation}
    $\{a_i\}_{i=1}^3$ are called energy wells for the potential $W$. At each $a_i$, the Hessian matrix $D^2W(a_i)$ is positive definite. Moreover, $u\cdot W_u(u)>0$ for $|u|\geq 2$.
    \item[(P2).] Let $D_3$ denote the group of symmetries of the equilateral triangle, consisting of rotations by multiples of $\frac{2\pi}{3}$ about the origin and the reflections across lines $\{(x_1,x_2): x_2=0\}$, $\{(x_1,x_2): x_2=\sqrt{3}x_1\}$ and $\{(x_1,x_2):x_2=-\sqrt{3}x_1\}$. $W$ satisfies 
    \begin{equation*}
        W(g \bx)=W(\bx), \quad \forall \bx\in \mathbb{R}^2,\, g\in D_3.
    \end{equation*}
\end{enumerate}

For this specific choice of $W$, system \eqref{EL eq} can be rewritten as 
\begin{align}
    \label{eq u1} &\Delta u_1=4u_1(u_1^2+u_2^2)+2u_2^2-2u_1^2-2u_1,\\
    \label{eq u2} &\Delta u_2=4u_2(u_1^2+u_2^2)+4u_1u_2-2u_2.
\end{align}

We next specify the class of entire solutions for the system \eqref{EL eq} considered in this paper. 

\begin{definition}
    $u\in C^\infty(\mathbb{R}^2;\mathbb{R}^2)$ is called an \emph{entire minimizing solution} to \eqref{EL eq} if it satisfies \eqref{EL eq} in $\BR^2$, and for any bounded open set $K\Subset \mathbb{R}^2$ and every $v\in H_0^1(K;\mathbb{R}^2)$, 
    \begin{equation*}
        J(u,K)\leq J(u+v, K).
    \end{equation*}
\end{definition}

Our main result gives a complete classification of bounded, entire minimizing solutions to \eqref{EL eq}. 
\begin{theorem}\label{main thm}
    Suppose $u$ is a bounded, entire minimizing solution to \eqref{EL eq} with $W$ given by \eqref{potential specific}. Then exactly one of the following holds:
    \begin{itemize}
        \item[1.] Trivial constant solution. There exists $i\in\{1,2,3\}$ such that 
        \begin{equation*}
            u(\bx)\equiv a_i,\quad \forall \bx\in \mathbb{R}^2.
        \end{equation*}
        \item[2.] Two-phase solution with one-dimensional symmetry. There exist $i,\,j\in\{1,2,3\}$, $i\neq j$, $h\in \mathbb{R}$, $\mathbf{n}\in \mathbb{S}^1$, such that 
        \begin{equation}\label{2 phase sol}
            u(\mathbf{x})=  \frac{e^{\frac{\sqrt{6}}{2}(\mathbf{x}\cdot \mathbf{n}-h)} }{e^{\frac{\sqrt{6}}{2}(\mathbf{x}\cdot \mathbf{n}-h)}+e^{-\frac{\sqrt{6}}{2}(\mathbf{x}\cdot \mathbf{n}-h)}} \ a_i+ \frac{e^{-\frac{\sqrt{6}}{2}(\mathbf{x}\cdot \mathbf{n}-h)} }{e^{\frac{\sqrt{6}}{2}(\mathbf{x}\cdot \mathbf{n}-h)}+e^{-\frac{\sqrt{6}}{2}(\mathbf{x}\cdot \mathbf{n}-h)}} \ a_j.
        \end{equation}
        In other words, $u$ is a one-dimensional solution depending only on one scalar variable $\bx\cdot\mathbf{n}$. 
        \item[3.] Triple-junction solution. There exist $\bx_0\in \mathbb{R}^2$ and $Q\in O(2)$ such that 
        \begin{equation}
            \label{triple sol}
            u(\bx)= \sum_{i=1}^3 \frac{e^{\sqrt{2} a_i\cdot (Q \bx-\bx_0)}}{ \sum_{j=1}^3 e^{\sqrt{2} a_j\cdot (Q \bx-\bx_0)}}\ a_i,
        \end{equation}
        where $Q\in O(2)$ is an orthogonal transformation. 
    \end{itemize}
\end{theorem}

\begin{rmk}
    It is classical that the one-phase solution must be constant. The one-dimensional symmetry for the two-phase minimizing solutions is also previously known, see \cite[Theorem 1.3]{ss2024}. But the explicit formula \eqref{2 phase sol} for the specific potential \eqref{potential specific} is first derived here. The major novelty and highlight of the theorem is the unique identification of the triple-junction solution \eqref{triple sol}.
\end{rmk}

Understanding solutions of Allen-Cahn equation and system has been a central topic in PDE, geometric analysis and calculus of variations. The classical model is the scalar Allen-Cahn equation
\begin{equation}\label{scalar ac}
    \Delta u +u-u^3=0, \quad u: \mathbb{R}^N\to \mathbb{R},
\end{equation}
which is the Euler-Lagrange equation of the energy functional
\begin{equation*}
    E(u,\Om)=\int_\Om \f12|\na u|^2 +\f{(1-u^2)^2}{4} \,d\bx.
\end{equation*}

The celebrated conjecture of De Giorgi \cite{Degiorgi1978} states the following: let $u$ be a bounded entire solution of \eqref{scalar ac} which is monotone in one direction, say $\frac{\pa u}{\pa x_N}>0$. Then when $N\leq 8$, $u$ is one-dimensional. In particular, $u$ must be of the form $u(\bx)=\tanh{\left( \f{\bx\cdot\nu-c}{\sqrt2} \right)}$, for some $c\in \BR$, $\nu\in \mathbb{S}^{N-1}$. This conjecture shows a deep connection between Allen-Cahn solutions and minimal surfaces. Indeed, the threshold dimension $N=8$ here corresponds to the critical dimension in the Bernstein problem for minimal graphs. De Giorgi's conjecture has not been proved in its full generality. It was established in dimension 2 by Ghoussoub-Gui \cite{ghoussoub1998conjecture} and in dimension 3 by Ambrosio-Cabr\'{e} \cite{ambrosio2000entire}. For $4\leq N\leq 8$, Savin \cite{savin2009regularity} proved the conjecture under an extra condition on the limit of $u$ as $x_N\to\pm\infty$. Counterexamples in dimensions $N\geq 9$ were constructed by del Pino-Kowalczyk-Wei in \cite{del2011giorgi}. If one replaces the monotonicity-in-one-direction condition by minimality, the one-dimensional symmetry of entire minimizers was established by Savin for $N\leq 7$ \cite{savin2009regularity}, while counterexamples were constructed by Liu-Wang-Wei in \cite{liu2017global} for $N\geq 8$. 

When the potential $W$ has more than two energy wells, vector-valued order parameters arise naturally to describe multiphase transitions. And minimizing solutions are related to minimal partitions through the $\Gamma$-convergence technique, see e.g. \cite{sternberg1988effect,fonseca1989gradient,baldo}. As for entire solutions, which can be considered as asymptotic profiles near diffuse transition interfaces, one expects solutions with a multiphase junction structure to appear, whose sharp interface limits correspond to Almgren minimal cones. Moreover, even for double-well potentials, new profiles will also appear in vectorial setting due to the possible non-uniqueness of the heteroclinic connections between two phases. For example, Alama-Bronsard-Gui \cite{alama1997stationary} and Schatzman \cite{schatzman2002asymmetric} constructed non-one-dimensional entire minimizing solutions in $\BR^2$ which connect two different heteroclinic connections for a double-well potential.

A particularly interesting case is the genuinely two-dimensional profile when three phases meet at a codimension-2 junction set, which corresponds to a $Y$-shaped minimal cone. The first existence result for such a triple-junction solution in $\BR^2$ was obtained by Bronsard-Gui-Schatzman in \cite{bronsard1996three}, under the $D_3$-symmetry assumption on both $W$ and $u$. In their construction, $W$ satisfies (P1) and (P2), and $u$ is restricted to the equivariant class $u(g\,x)=g\,u(x)$ for $g\in D_3$. In the more general setting with higher dimensions, more energy wells and general finite symmetry groups, the existence of symmetric minimizing solutions was established by Alikakos-Fusco in \cite{alikakos2011entire}. More recently, Fusco \cite{fusco2024minimizing} proved a similar existence result under the rotational symmetry alone, thus eliminating the reflection symmetry requirements. All these constructions provide critical points of the energy functional, but stability under non-symmetric perturbations remains open. 

There are several recent developments on the minimizing triple-junction solutions without imposing symmetry hypotheses on the potential or the solution. Alikakos and the author \cite{alikakos2024triple}, Sandier and Sternberg \cite{ss2024} independently proved the existence of an entire minimizing solution with a triple-junction blow-down limit along a sequence $R_k\to\infty$. Following these works, the author \cite{geng2025uniqueness,geng2025rigidity} further proved the uniqueness of the blow-down limit as well as the almost one-dimensional symmetry along each leg of the interface under the extra assumption on the uniqueness of heteroclinic connections. We will discuss these results in greater detail in Section \ref{subsec: asymp}.

This leads to a natural rigidity question: \emph{if $W$ is $D_3$-invariant, must every entire minimizing solution with triple-junction asymptotics inherit the $D_3$-equivariance symmetry?} Under assumptions (P1) and (P2), two a priori different constructions are available. Bronsard-Gui-Schatzman \cite{bronsard1996three} and Alikakos-Fusco \cite{alikakos2011entire} construct entire solutions by solving constrained minimization problems within the $D_3$-equivariant class, whereas the existence results without symmetry constraints \cite{alikakos2024triple,ss2024} produce entire solutions that minimize with respect to arbitrary compactly supported perturbations. These two types of solutions have the same triple-junction asymptotic structure at infinity. However, symmetry of the energy functional does not in general force symmetry of the minimizer, and identical asymptotic behavior alone does not imply uniqueness. It is therefore very natural to ask whether these two solutions coincide, up to a translation and orthogonal transformation of the coordinates.

This paper answers these questions affirmatively for the specific potential $W$ given by \eqref{potential specific}. More precisely, Theorem \ref{main thm} classifies every entire minimizing solution with a triple-junction structure at infinity: after a suitable translation and orthogonal change of coordinates, it is given by the explicit profile $u_*$ \eqref{map:u*}. In particular, every such solution is $D_3$-equivariant with respect to a shifted center. The same conclusion holds for every $D_3$-equivariant minimizer arising from the construction of \cite{bronsard1996three,alikakos2011entire}, where the center is automatically fixed at the origin. Thus, the constrained and unconstrained variational constructions are always the same solution. To the best of our knowledge, this is the first complete classification of entire triple-junction minimizers for a genuinely vector-valued triple-well Allen-Cahn system. The explicit and rigid nature of \eqref{triple sol} also suggests applications in higher dimensions. For $N\geq 3$, it may serve as a canonical building block for asymptotic profiles near the ternary sharp interface, much as the one-dimensional heteroclinic $u(x)=\tanh(\f{x}{\sqrt{2}})$ does in the scalar Allen-Cahn theory. In particular, the cylindrical extension $U(x_1,\ldots,x_N)=u(x_1,x_2)$ provides a natural starting point for studying the rigidity of entire minimizing triple-junction solutions with $N-2$ invariant directions in higher dimensional space, with the specific potential \eqref{potential specific}. 

We note that the explicit profile $u_*$ in \eqref{map:u*} has appeared in the physics literature. After a suitable normalization, it coincides with an exact domain-wall BPS junction constructed by Kakimoto-Sakai \cite{kakimoto2003domain} and later included in a general family of symmetric domain-wall junctions by Eto-Kawaguchi-Nitta-Sasaki \cite{eto2020exact}. Those works construct particular solutions for a first-order BPS system, while the current paper demonstrates that $u_*$, up to the natural symmetries, is the unique entire Allen–Cahn minimizer with triple-junction asymptotics at infinity.

The potential \eqref{potential specific} was first constructed by Alikakos-Fusco in \cite{AlikakosFusco2009}. They considered the general family of quartic polynomials which are $D_3$-invariant:
\begin{equation*}
     W(u)=\al R^4+\beta R^3\cos{3\theta}+\gamma R^2+\delta,\quad\alpha>0,\,\beta,\g,\delta\in\mathbb{R}, \ \  u=(R\cos\theta,R\sin\theta).
\end{equation*}
The particular choice \eqref{potential specific} is canonical in the sense that after the normalization $\al=1$, $\beta,\g,\delta$ are uniquely determined by the conditions
\begin{align*}
    &\qquad\qquad\quad W(a_1)=0,\quad W_u(a_1)=0,\\
    &s\to W(a_1+s\nu) \text{ is nondecreasing as long as } a_1+s\nu \in \Omega_1,
\end{align*}
where $\nu\in\mathbb{S}^1$, and $\Omega_i$ is defined by \eqref{def Omi}. Here the radial-monotonicity of $W$ around $a_i$ in each $\Om_i$ is a special case (and the most natural case) for the technical Hypothesis 4 in \cite{AlikakosFusco2009,alikakos2011entire}, which is essential in their construction to provide necessary variational estimates to obtain a nontrivial entire solution as a limit of minimizers for a constrained minimizing problem of the Allen-Cahn functional on $B_R$ as $R\to\infty$. In the present work, we uncover additional algebraic structures specific to this potential that yield a much stronger conclusion, namely, the uniqueness of entire minimizing solutions with triple-junction asymptotics, up to translations and orthogonal transformations. It would be interesting to explore a broader class of potentials for which an analogous rigidity result holds for entire multiphase minimizers.

To illustrate our idea, we first examine the one-dimensional heteroclinic connection $U(x)=\tanh\f{x}{\sqrt2}$ for the scalar Allen-Cahn equation
\begin{equation*}
    U''+U-U^3=0,\quad U(+\infty)=1,\,U(-\infty)=-1.
\end{equation*}
Multiplying the equation by $U'$, we get the conservation law
\begin{equation*}
    \frac{d}{dx} \left(\f{|U'|^2}{2}-\f{(1-|U|^2)^2}{4}\right)=0. 
\end{equation*}
For a finite energy profile, $\f{|U'|^2}{2}-\f{(1-|U|^2)^2}{4}$ needs to vanish, which yields the first-order equation
\begin{equation}\label{1st eq scalar}
    U'=\pm\sqrt{2W(U)}=\pm\frac{1-U^2}{\sqrt2}. 
\end{equation}
From the limits of $U$ at infinity, we choose $U'>0$ and integrate the equation \eqref{1st eq scalar} using separation of variables and obtain $U(x)=\tanh(\f{x}{\sqrt{2}})$, which is unique up to translations.

Now we look at the two-dimensional problem from the same perspective. The preceding computation shows that the essential source of the rigidity for one-dimensional heteroclinic connection is the reduction of the second-order Euler-Lagrange equation to a first-order equation. Remarkably, the potential \eqref{potential specific} admits an exact two-dimensional analogue of this reduction. More precisely, there exists an explicitly defined matrix field $\mathcal{N}:\mathbb{R}^2\to\mathbb{R}^{2\times 2}$, given in \eqref{def of N}, such that
\begin{equation*}
  W(u)=\frac12|\mathcal{N}(u)|^2+\frac13\det \mathcal{N}(u).  
\end{equation*}
Moreover, the Allen-Cahn energy density satisfies the following inequality
\begin{equation}\label{intro:cal}
\frac12|\nabla u|^2+W(u)\geq  \text{ null-Lagrangian terms }+\frac13|\nabla u-\mathcal{N}(u)|^2.
\end{equation}
Thus $\mathcal{N}(u)$ provides a two-dimensional counterpart of the term $\sqrt{2W(U)}$ as in scalar problem. This algebraic compatibility between $W$ and $\mathcal{N}$ is the main source of the rigidity of \eqref{potential specific}.

For an arbitrary minimizing solution with triple-junction asymptotics, the results recalled in Section \ref{subsec: asymp} show that along the directions $\bt_{ij}$ of three interfaces, $u$ converges to translated heteroclinic connections $U_{ij}( \bx\cdot \bt_{ij}^\perp-h_{ij})$. A first-moment identity for the stress tensor yields the compatibility condition
$$
h_{12}+h_{23}+h_{31}=0,
$$
which implies that the three shifts arise from a single translation of the center by $\mathbf{z}\in\mathbb{R}^2$. Hence $u$ and the translated explicit solution $u_*(\cdot-\mathbf{z})$ have asymptotically identical boundary values on large equilateral triangles. Comparing their energies, using the minimality of $u$ and the inequality \eqref{intro:cal} above, forces the nonnegative defect $|\nabla u-\mathcal{N}(u)|^2$ to vanish identically. It remains only to integrate this first-order system and obtain the explicit formula for $u$. 

The proof is organized accordingly. In Section \ref{sec:pre}, we determine the minimizing heteroclinic connections and recall the classification and asymptotic behavior of blow-down limits for entire solutions, which also settle the one-phase and two-phase cases. In Section \ref{sec:u*}, we construct the explicit triple-junction solution $u_*$ and verify its symmetry, first-order and second-order equations, and asymptotics at infinity. Finally, Section \ref{sec:proof} establishes the compatibility of the three shifts $h_{ij}$, derives the first-order system $\na u=\mathcal{N}(u)$, and integrates that system to complete the proof of Theorem \ref{main thm}.

\section{Preliminaries}\label{sec:pre}

\subsection{Heteroclinic connection.} For $i,j\in \{1,2,3\}$, we need to understand the behavior of $u$ across the interface separating the two phases $a_i$ and $a_j$. We call $U_{ij}\in W^{1,2}_{loc}(\mathbb{R},\mathbb{R}^2)$ a \emph{heteroclinic connection} between $a_i$ and $a_j$ if $U_{ij}$ solves the following variational problem:
\begin{align*}
    &U_{ij} \text{ minimizes }\int_\mathbb{R} \left(\f12| u'|^2+W(u)\right)\,dx,\\
    &\lim\limits_{x\to-\infty} U_{ij}(x)=a_i,\quad \lim\limits_{x\to\infty} U_{ij}(x)=a_j.
\end{align*}
The existence of heteroclinic connections for any $i\neq j$ is classical, see for instance \cite{alikakos2008connection,zuniga2016heteroclinic}. 

For $a_3=(-\f12,-\frac{\sqrt{3}}{2}),\, a_2=(-\f12,\frac{\sqrt{3}}{2})$, we derive an explicit formula for the heteroclinic connection $U_{32}$. For convenience, we write $U=U_{32}$. $U=(U_1,U_2)$ satisfies
\begin{equation}\label{ode:connection}
\begin{split}
    &U_1''= 4U_1(U_1^2+U_2^2)+2U_2^2-2U_1^2-2U_1,\quad U_1(-\infty)=U_1(+\infty)=-\f12,\\
    &U_2''=4U_2(U_1^2+U_2^2)+4U_1U_2-2U_2,\quad U_2(-\infty)=-\f{\sqrt{3}}{2},\  U_2(+\infty)=\frac{\sqrt{3}}{2}.
\end{split}
\end{equation}

Define new variables
\begin{equation*}
    s=U_1+\f12,\quad v=U_2,
\end{equation*}
then the energy functional becomes
\begin{equation}\label{ene s v}
    E(s,v)=\int_{\mathbb{R}} \left(\f12|s'|^2+\f12|v'|^2+ s^2(s^2-\f83 s +\f32+2v^2)+(v^2-\f34)^2\right), 
\end{equation}
with the associated Euler-Lagrange system
\begin{equation}\label{ode:connection 2}
\begin{split}
    &s''=s(4(s-1)^2+4v^2-1),\quad s(-\infty)=s(+\infty)=0,\\
    &v''=v(4v^2+4s^2-3),\quad v(-\infty)=-\f{\sqrt{3}}{2},\  v(+\infty)=\frac{\sqrt{3}}{2}.
\end{split}
\end{equation}

Observe that when $s\equiv 0 $, or equivalently $U_1\equiv -\f12$, the ODE for $s$ is trivially satisfied. Then the energy for $v_2$ becomes the classical scalar Allen-Cahn functional $\int_{\mathbb{R}} \f12|v'|^2+(v^2-\f34)^2\,dx$, 
whose classical heteroclinic connection is given by 
\begin{equation*}
    v(x)=\f{\sqrt{3}}{2} \tanh{(\f{\sqrt{6}}{2}x)}.
\end{equation*}
Therefore we obtain a solution to the ODE system \eqref{ode:connection}:
\begin{equation}
    U=\left(-\f12, \f{\sqrt{3}}{2} \tanh{(\f{\sqrt{6}}{2}x)}\right).
\end{equation}
It remains to verify that it is the unique minimizer.

\begin{lemma}\label{lem: 1d connection}
    Up to translation in $x$, $U=(-\f12, \f{\sqrt{3}}{2} \tanh(\f{\sqrt{6}}{2}x))$ is the only (minimizing) heteroclinic connection between $(-\f12,\pm\f{\sqrt{3}}{2})$.
\end{lemma}
\begin{proof}
    Assume $U=(U_1,U_2)$ is a minimizing heteroclinic connection. It suffices to show that $U_1\equiv -\f12$. By symmetry, we may assume that the connection $\{U(x): x\in\mathbb{R}\}$ is entirely contained in the sector  $\{(r\cos\varphi,r\sin\varphi):\varphi\in[\f{2\pi}3,\f{4\pi}{3}]\}$. Indeed, if any portion of the curve lies outside this sector, then by applying successive reflections across the lines $\{x_2=-\sqrt{3}x_1\}$ and $\{x_2=\sqrt{3}x_1\}$, we can fold the curve back into the sector without increasing the total energy. Consequently, for every $x\in \mathbb{R}$, we may assume $s(x)=U_1(x)+\f12\leq 0$.

    Examining the energy \eqref{ene s v}, we consider the natural competitor $(0,v)$ and compute the energy difference,
    \begin{equation*}
        E(s,v)-E(0,v)=\int_{\mathbb{R}} \left( \f12|s'|^2+s^2(s^2-\f83 s+\f32+2v^2) \right)\,dx \geq \int_{\mathbb{R}} \f12|s'|^2+\f{5}{12} s^2\,dx, 
    \end{equation*}
    where we used $s^2-\f83 s+\f32\geq \f{5}{12}$ when $s\leq \f12$. Therefore, truncation of $s$ to $0$ will strictly reduce the total energy. Once we fix $s\equiv 0$, the remaining component $v$ must be the unique scalar connection up to translation. The proof is complete.
    
\end{proof}

Finally, we rewrite $U_{32}$ as
\begin{equation}\label{formula: U32}
    U_{32}(x)=\frac{e^{\f{\sqrt{6}}{2}x}}{e^{\f{\sqrt{6}}{2}x}+e^{-\f{\sqrt{6}}{2}x}}\ a_2+ \frac{e^{-\f{\sqrt{6}}{2}x}}{e^{\f{\sqrt{6}}{2}x}+e^{-\f{\sqrt{6}}{2}x}} \ a_3,\quad \text{up to translation}.
\end{equation}
Using the $D_3$-symmetry, we have for any $i\neq j\in \{1,2,3\}$,
\begin{equation}\label{formula: 1d connection}
    U_{ij}(x)=\frac{e^{\f{\sqrt{6}}{2}x}}{e^{\f{\sqrt{6}}{2}x}+e^{-\f{\sqrt{6}}{2}x}}\ a_j+ \frac{e^{-\f{\sqrt{6}}{2}x}}{e^{\f{\sqrt{6}}{2}x}+e^{-\f{\sqrt{6}}{2}x}} \ a_i,\quad \text{up to translation.}
\end{equation}

\subsection{Asymptotic behavior of $u$ at infinity.}\label{subsec: asymp}

The entire minimizing solution $u$ to the Allen-Cahn system \eqref{EL eq} is closely related to the minimal partition problem. More precisely, we have the following characterization of its possible blow-down limits at infinity, see e.g. \cite{alikakos2013structure,ss2024,alikakos2024triple}.

\begin{proposition}
Let $u$ be a bounded entire minimizing solution of \eqref{EL eq}. Then there exists a sequence $r_k\to \infty$ such that
\begin{equation}\label{L1 conv to blowdown}
    u(r_k\bx)\xrightarrow[k\to\infty]{L^1_{loc}(\mathbb{R}^2)} u_0,
\end{equation}
where $u_0\in \mathrm{BV}_{loc}(\mathbb{R}^2; \{a_1,a_2,a_3\})$ has one of the three following forms:
\begin{itemize}
    \item[1.] \emph{One-phase profile. }$u_0\equiv a_i$ for some $i\in\{1,2,3\}$.
    \item[2.] \emph{Two-phase profile. }$u_0=a_i\mathbf{1}_{D}+a_j \mathbf{1}_{\mathbb{R}^2\setminus D}$, where $i\neq j\in \{1,2,3\}$ and $D$ is a half-plane bounded by a line through the origin. 
    \item[3.] \emph{Triple-junction profile. }$u_0=\sum_{i=1}^3 a_i \mathbf{1}_{\Om_i}$, where $\{\Om_1,\Om_2,\Om_3\}$ partitions $\mathbb{R}^2$ into three infinite sectors, each with opening angle $\f{2\pi}3$ and vertex at the origin.
\end{itemize}

\end{proposition}

We now examine three possible blow-down limits separately. 

\begin{proposition}\label{prop: 1 phase}
    Let $u$ be a bounded entire minimizing solution of \eqref{EL eq} and $u_0$ be a blow-down limit of $u(r_k\bx)$ given by \eqref{L1 conv to blowdown}. If $u_0\equiv a_i$, then $u(\bx)\equiv a_i$ in $\BR^2$. 
\end{proposition}
This Liouville-type result has been mentioned in several works. It follows directly from the following vector version of the Caffarelli-C\'{o}rdoba density estimate \cite{AF3}.
\begin{lemma}\label{lem: caf-cor}
    Let $\Omega\subset \BR^2$. Suppose $u\in W^{1,2}(\Omega,\BR^2)\cap L^\infty(\Om,\BR^2)$ is a minimizing solution of \eqref{EL eq}.  If for some $r_0,\,\lam,\,\mu_0>0$, $x\in\Omega$, $i\in \{1,2,3\}$,
    \beqo
    \mathcal{L}^2(B(x,r_0)\cap \{\vert u-a_i\vert>\lam\})\geq \mu_0,
    \eeqo
    then there exists a constant $C(\mu_0,\lam,r_0)>0$ such that
    \beqo
    \mathcal{L}^2(B(x,r)\cap \{|u-a_i|>\lam\})\geq C(\mu_0,\lam,r_0) r^2, \quad \forall r\geq r_0,\ B(x,r)\Subset \Om.
    \eeqo
\end{lemma}

If a blow-down limit $u_0$ has a two-phase profile, we can invoke the following rigidity result \cite{ss2024}  for two-phase minimizer thanks to the fact that 1D heteroclinic connections $U_{ij}$ are unique up to translation. 

\begin{proposition}{\cite[Theorem 1.3]{ss2024}}\label{prop: 2 phase}
Assume $u$ is a bounded, entire minimizing solution of \eqref{EL eq} and $u(r_k\bx)$ converges in $L^1_{loc}$ to $u_0$ along a subsequence $r_k\to\infty$, where   $u_0=a_i\mathbf{1}_{D}+a_j\mathbf{1}_{\mathbb{R}^2\setminus D}$ is a two-phase blow-down limit. Then there exists $h\in\mathbb{R}$ such that 
\begin{equation*}
    u(\bx)=U_{ji}(\bx\cdot \mathbf{n}-h).
\end{equation*}
Here $\mathbf{n}$ is the inward unit normal on $\pa D$. 
\end{proposition}

We are primarily interested in the third case, where along a sequence $r_k\to \infty$, $u(r_k\cdot)$ converges in $L^1_{loc}$ to a triple-junction map which connects all three phases $\{a_1,a_2,a_3\}$. The existence of such entire minimizing solution was established independently in \cite{alikakos2024triple} and \cite{ss2024}. The uniqueness of the blow-down limit and the almost 1D symmetry along each leg of the sharp interface were subsequently proved in \cite{geng2025uniqueness} and \cite{geng2025rigidity}. We recall these results below.

\begin{proposition}\label{prop:rigidity for triple junction}
Let $u$ be an entire minimizing solution and suppose that along a sequence $r_k\to\infty$, $u(r_k\bx)$ converges in $L^1_{loc}$ to a triple-junction map $u_0=\sum_{i=1}^3 a_i\mathbf{1}_{\Om_i}$, where $\{\Om_i\}$ partitions $\mathbb{R}^2$ into three sectors, each with opening angle $\frac{2\pi}{3}$ and vertex at the origin. Then the convergence is independent of the sequence $\{r_k\}$, that is,
\begin{equation}\label{pre:L1 conv}
    u(r\bx) \xrightarrow[r\to\infty]{L^1_{loc}} u_0.
\end{equation}
For distinct $i,j\in\{1,2,3\}$, let $\boldsymbol{\tau}_{ij}$ denote the unit tangent vector to the sharp interface ray $\pa \Om_i\cap \pa \Om_j$, and let $\mathbf{n}_{ij}$ denote the unit normal vector of $\pa \Om_i\cap\pa \Om_j$ pointing from $\Om_i$ into $\Om_j$. Then there exists $h_{ij}\in\mathbb{R}$ such that 
\begin{equation}\label{1d asymptotic}
     u(x_1\bt_{ij}+x_2\bn_{ij}) \xrightarrow[x_1\to+\infty]{C^{2,\alpha}(\mathbb{R},\mathbb{R}^2)} U_{ij}(x_2-h_{ij}).
\end{equation}
Moreover, there are constants $C,k>0$, depending on $W,u$, such that 
\begin{equation}\label{exp decay}
    |u(\bx)-a_i|\leq Ce^{-k\,\dist(\bx,\pa \Om_i)}, \quad \forall \bx\in \Om_i, \ i\in\{1,2,3\}.
\end{equation}
\end{proposition}

\begin{rmk}\label{rmk: exp decay for der}
    Combining \eqref{exp decay} with standard elliptic estimates, we can get similar exponential decay for all derivatives of $u$. More precisely, for any $m\geq 1$, there exists $C_m,k_m>0$ such that 
    \begin{equation*}
        |D^m u(\bx)| \leq C_m e^{-k_m\mathrm{dist}(\bx,\pa\Om_i)},\quad \bx\in \Om_i.
    \end{equation*}
\end{rmk}

\section{An explicit triple-junction solution}\label{sec:u*}

We begin with the explicit form of the heteroclinic connection $U_{32}$. If we put the connection on the $x_2$-axis of $\mathbf{R}^2$ and extend it constantly in the $x_1$-direction, then the formula \eqref{formula: U32} becomes
\begin{equation*}
    U_{32}(\bx) = \frac{e^{\sqrt{2}a_2\cdot  \bx}}{e^{\sqrt{2}a_2\cdot  \bx}+e^{\sqrt{2}a_3\cdot  \bx}}\ a_2+ \frac{e^{\sqrt{2}a_3\cdot  \bx}}{e^{\sqrt{2}a_2\cdot  \bx}+e^{\sqrt{2}a_3\cdot  \bx}} \ a_3.
\end{equation*}
Motivated by this, we consider the following map
\begin{equation}
    \label{map:u*}
    u_*(\bx):= \frac{e^{\sqrt{2}a_1\cdot  \bx}}{\sum\limits_{i=1}^3 e^{\sqrt{2}a_i\cdot \bx}}\ a_1+\frac{e^{\sqrt{2}a_2\cdot  \bx}}{\sum\limits_{i=1}^3 e^{\sqrt{2}a_i\cdot \bx}}\ a_2+\frac{e^{\sqrt{2}a_3\cdot  \bx}}{\sum\limits_{i=1}^3 e^{\sqrt{2}a_i\cdot \bx}}\ a_3. 
\end{equation}
We shall prove that $u_*\in C^\infty(\BR^2;\BR^2)$ is a solution to \eqref{EL eq} and has the triple-junction asymptotics described in Proposition \ref{prop:rigidity for triple junction}. We first introduce the following notation which will be used throughout the rest of the paper.
\begin{align}
\label{def Omi}&\Om_i:=\{(r\cos\theta,r\sin\theta):\  0<r<\infty,\ \frac{(2i-3)\pi}{3}<\theta<\frac{(2i-1)\pi}{3}\},\quad i\in\{1,2,3\}. \\ 
   \nonumber &\qquad\qquad\qquad\bt_{12}:= (\f12,\f{\sqrt3}{2}),\quad \bn_{12}=(-\f{\sqrt3}{2},\f12),\\
   \nonumber &\qquad\qquad\qquad\bt_{23}:= (-1,0),\quad \bn_{23}=(0,-1),\\
   \nonumber &\qquad\qquad\qquad\bt_{31}:= (\f12,-\f{\sqrt3}{2}),\quad \bn_{31}=(\f{\sqrt3}{2},\f12).
\end{align}
By construction, $\bt_{ij}$ represents the direction of $\pa \Om_i\cap\pa \Om_j$ and $\bn_{ij}$ is the unit normal on $\pa \Om_i\cap\pa \Om_j$ pointing from $\Om_i$ into $\Om_j$.  

\begin{proposition}\label{prop: u* sol}
    The map $u_*$ solves \eqref{EL eq}. Moreover, it is $D_3$-equivariant: 
    \begin{equation*}
        u_*(g\, \bx)=g\, u_*(\bx),\quad \forall g\in D_3,\ \bx\in\BR^2.
    \end{equation*}
    It also satisfies
    \begin{equation}\label{L1 conv in prop u* sol}
        u_*(r\bx)\xrightarrow[r\to\infty]{L_{loc}^1} \sum_{i=1}^3 a_i\mathbf{1}_{\Om_i},    
    \end{equation}
    and there are constants $C,k>0$ such that 
    \begin{equation}\label{exp close in prop u* sol}
        |u_*(\bx)-a_i|\leq Ce^{-k\,\dist(\bx,\pa \Om_i)}, \quad \forall \bx\in \Om_i, \ i\in\{1,2,3\}.
    \end{equation}
\end{proposition}

\begin{proof}
    Define 
    \begin{equation*}
        S(\bx):= \sum\limits_{i=1}^3 e^{\sqrt{2}a_i\cdot \bx}, \quad \lambda_i(\bx):= \frac{e^{\sqrt{2}a_i\cdot\bx}}{S(\bx)},\quad i=1,2,3.
    \end{equation*}
    Then $\lambda_i(\bx)>0$, $\sum_{i=1}^3\lambda_i(\bx)=1$ and $u_*(\bx)=\sum_{i=1}^3 \lambda_i(\bx)a_i$. We divide the proof into three steps.

    \begin{step}
        $u_*$ satisfies the $D_3$-symmetry. For any $g\in D_3$, it holds that 
        \begin{align*}
            u_*(g\, \bx)&=\sum_{i=1}^3 \frac{e^{\sqrt2 a_i\cdot g\bx}}{S(\bx)}a_i\\
            =\sum_{i=1}^3 & \frac{e^{\sqrt2(g^{-1} a_i)\cdot x}}{S(\bx)} a_i= \sum_{i=1}^3 \f{e^{\sqrt2 a_i\cdot \bx}}{S(\bx)} g\, a_i=g\, u_*(\bx).
        \end{align*}
    \end{step}

    \begin{step}
        $u_*$ solves \eqref{EL eq}. Using the elementary identities
        \begin{equation*}
            \nabla (e^{\sqrt2 a_i\cdot \bx})=\sqrt2 a_i e^{\sqrt2 a_i\cdot \bx},\quad \nabla S(\bx)=\sqrt{2}S(\bx)u_*(\bx),
        \end{equation*}
        we get
        \begin{align*}
            \nabla u_*& = \sum_{i=1}^3 a_i \otimes \frac{\sqrt2 a_i e^{\sqrt2 a_i\cdot \bx} S(\bx)-\sqrt{2}S(\bx)u_*(\bx) e^{\sqrt2 a_i\cdot\bx}}{S(\bx)^2}\\
            &=\sqrt2\left( \sum_{i=1}^3 \lambda_i(\bx) a_i\otimes a_i -u_*(\bx)\otimes u_*(\bx)  \right).
        \end{align*}
        By \eqref{coord: ai}, we have 
        \begin{equation}\label{ai otimes}
            a_i\otimes a_i= \frac{1}{2}\begin{pmatrix}
                1+a_i^1 & -a_i^2\\
                -a_i^2 & 1-a_i^1
            \end{pmatrix},\quad \text{where } a_i=(a_i^1,a_i^2),\ i\in\{1,2,3\}.
        \end{equation}
        Therefore, writing $u_*=(u_1,u_2)$, we have 
        \begin{equation}\label{nabla u}
            \nabla u_*= \sqrt{2} \begin{pmatrix}
                \f{1+u_1}{2}-u_1^2 & -\f{u_2}{2}-u_1u_2\\
                -\f{u_2}{2}-u_1u_2 & \f{1-u_1}{2}-u_2^2
            \end{pmatrix}. 
        \end{equation}
    Differentiating $Du_*$ gives
        \begin{align*}
            \Delta u_1 &= \pa_1\left(\sqrt{2}(\f{1+u_1}{2}-u_1^2)\right)+\pa_2\left( \sqrt{2}(-\frac{u_2}{2}-u_1u_2) \right)\\
            &=\frac{\sqrt2}{2} \pa_1u_1-2\sqrt2 u_1\pa_1u_1-\frac{\sqrt2}{2} \pa_2u_2-\sqrt2u_1\pa_2u_2-\sqrt2 u_2\pa_2u_1\\
            &=4u_1^3+4u_1u_2^2+2u_2^2-2u_1^2-2u_1;
        \end{align*}
        \begin{align*}
            \Delta u_2&= \pa_1\left( \sqrt2(-\frac{u_2}{2}-u_1u_2)\right)+\pa_2\left( \sqrt2(\frac{1-u_1}{2}-u_2^2) \right)\\
            &=-\frac{\sqrt2}{2} \pa_1 u_2-\sqrt2 u_1\pa_1u_2-\sqrt2u_2\pa_1u_1-\frac{\sqrt2}{2} \pa_2u_1-2\sqrt2 u_2\pa_2 u_2\\
            &=4u_2^3+4u_2u_1^2+4u_1u_2-2u_2.
        \end{align*}
    Thus $u_1,u_2$ satisfy \eqref{eq u1} and \eqref{eq u2} respectively, which is equivalent to saying that $u_*$ solves \eqref{EL eq}.    
    \end{step}
    \begin{step}
        To prove \eqref{L1 conv in prop u* sol} and \eqref{exp close in prop u* sol}, we need to demonstrate that in each $\Om_i$, $u_*$ converges to the corresponding phase $a_i$ with an exponential rate. By the $D_3$-equivariance of $u_*$, it suffices to show for $\bx=(r\cos\theta,r\sin\theta)$, where $\theta \in (0,\frac{\pi}{3})$,
        \begin{equation*}
            |u_*(\bx)-a_1|\leq Ce^{-k r \sin(\frac{\pi}{3}-\theta)}. 
        \end{equation*}
        Indeed, when $\theta\in (0,\frac{\pi}{3})$,
        \begin{align*}
            &\quad\bx\cdot a_1>\bx\cdot a_2>\bx\cdot a_3,\\
            &\bx\cdot (a_2-a_1)=-\sqrt{3} r\sin(\frac{\pi}{3}-\theta).
        \end{align*}
        By definition we have 
        \begin{align*}
            |u_*(\bx)-a_1|&\leq |\lambda_2(\bx)(a_2-a_1)|+|\lambda_3(\bx)(a_3-a_1)|\\
            &\leq 2\sqrt{3} |\lambda_2|\leq 2\sqrt{3} e^{\sqrt{2}\, \bx\cdot(a_2-a_1)}\leq 2\sqrt3 e^{-\sqrt6\, r\sin(\frac{\pi}{3}-\theta)}.
        \end{align*}
        This completes the proof.
    \end{step}
\end{proof}

\section{Proof of Theorem \ref{main thm}}\label{sec:proof}

Throughout this section, we assume $u$ is an entire minimizing solution with triple-junction asymptotics at infinity as described in Proposition \ref{prop:rigidity for triple junction}. The objective is to show $u=u_*$, up to orthogonal transformation and translation. 

After an orthogonal change of coordinates if necessary, we may assume without loss of generality that 
\begin{equation}\label{L1 conv triple}
    u(r\bx)\xrightarrow[r\to\infty]{L^1_{loc}} \sum_{i=1}^3 a_i\mathbf{1}_{\Om_i},\quad \Om_i\text{ given by \eqref{def Omi}}.
\end{equation}

According to Proposition \ref{prop:rigidity for triple junction}, for $(i,j)\in\{(1,2),(2,3),(3,1)\}$, there exists $h_{ij}\in \mathbb{R}$ such that  
\begin{equation}\label{1d symmetry one leg}
    u(x_1\bt_{ij}+x_2\bn_{ij})\xrightarrow[x_1\to\infty]{C^{2,\alpha}} U_{ij}(x_2-h_{ij}), \quad \forall \alpha \in(0,1),
\end{equation}
where $U_{ij}$ is given by \eqref{formula: 1d connection}. We begin by showing the following compatibility condition for $h_{ij}$.

\begin{lemma}\label{lem:h sum 0}
    $h_{12}+h_{23}+h_{31}=0$.
\end{lemma}

\begin{proof}
    We define the stress-energy tensor by
    \begin{equation}\label{def:str tensor}
        T_{\alpha\beta}=\pa_\alpha u\cdot \pa_\beta u-\delta_{\alpha\beta}\left(\f12|\na u|^2+W(u)\right), \quad \alpha,\beta\in\{1,2\}.
    \end{equation}
    Using the equation \eqref{EL eq}, we immediately get $\mathrm{div}(T)=0$, i.e. $\pa_\alpha T_{\alpha\beta}=0$. Let $X(\bx)=(-x_2,x_1)$. Since $T$ is symmetric and $\na X$ is antisymmetric, we have that 
    \begin{equation*}
        \mathrm{div}(TX)=\mathrm{div} T\cdot X+T:\na X=0.
    \end{equation*}
    Take $R\gg 1$, define the equilateral triangle
    \begin{equation}
        \mathcal{T}_R=\bigcap\limits_{(i,j)=(1,2),(2,3),(3,1)}\{\bx\cdot \bt_{ij}< R\}.
    \end{equation}
Its boundary consists of three line segments 
    \begin{equation*}
        l_{ij}=\{R\bt_{ij}+s\bn_{ij}: |s|<\sqrt3 R\},\quad (i,j)\in\{(1,2),(2,3),(3,1)\}.
    \end{equation*}
The outward unit normal to $l_{ij}$ is $\bt_{ij}$. Also, along $l_{ij}$,
\begin{equation*}
    X(R\bt_{ij}+s\bn_{ij})=R\bn_{ij}-s\bt_{ij}.
\end{equation*}

    We have 
    \begin{equation}\label{tensor comp}
    \begin{split}
        0&=\int_{\pa\mathcal{T}_R} \nu\cdot (TX)\,d\mathcal{H}^1\\
        &= \sum_{(i,j)}\int_{l_{ij}}\bt_{ij}^T TX\,d\mathcal{H}^1\\
        &= \sum_{(i,j)}\int_{-\sqrt3 R}^{\sqrt3 R} \bt_{ij}^T T (R\bn_{ij}-s\bt_{ij})\,ds\\
        &=\sum_{(i,j)}\Bigg[R\int_{-\sqrt3 R}^{\sqrt3 R} \bt_{ij}^T T \bn_{ij}\,ds-  \int_{-\sqrt3 R}^{\sqrt3 R}s \bt_{ij}^T T \bt_{ij}\,ds\Bigg],
        \end{split}
    \end{equation}
where the sum is taken over $(i,j)\in\{(1,2),(2,3),(3,1)\}$. 

    Set
    \begin{equation*}
        I^1_{ij}(R):= R\int_{-\sqrt3 R}^{\sqrt{3}R} \bt_{ij}^T T \bn_{ij}\,ds,\quad I^2_{ij}(R)= \int_{-\sqrt3 R}^{\sqrt3 R}s \bt_{ij}^T T \bt_{ij}\,ds.
    \end{equation*}
    We first estimate $I_{ij}^1(R)$. Since $\bt_{ij}\cdot \bn_{ij}=0$, we have  
    \begin{align*}
        I^1_{ij}(R)=R\int_{-\infty}^\infty \pa_{\bt_{ij}} u\cdot \pa_{\bn_{ij}} u\,ds- R\int_{|s|>\sqrt 3R} \pa_{\bt_{ij}} u\cdot \pa_{\bn_{ij}} u\,ds=-R\int_{|s|>\sqrt 3R} \pa_{\bt_{ij}} u\cdot \pa_{\bn_{ij}} u\,ds.
    \end{align*}
    Here we have utilized the important relation
    \begin{equation*}
        \int_{-\infty}^\infty \pa_{\bt_{ij}} u(R\bt_{ij}+s\bn_{ij})\cdot \pa_{\bn_{ij}} u(R\bt_{ij}+s\bn_{ij})\,ds=0,
    \end{equation*}
    which was established in \cite[Lemma 8.2]{schatzman2002asymmetric}. For $R$ sufficiently large, when $|s|>\sqrt{3}R$, $|\nabla u|$ can be controlled by $O(e^{-k|s|})$, therefore we obtain
    \begin{equation*}
        I^1_{ij}(R)\sim O(Re^{-kR})\to 0 \ \text{ as } R\to\infty.
    \end{equation*}

    For $I^2_{ij}(R)$, direct calculation gives
    \begin{equation*}
        I^2_{ij}(R)=\int_{-\sqrt3 R}^{\sqrt3 R} s \left( |\pa_{\bt_{ij}} u|^2-\f12|\na u|^2-W(u) \right)\,ds.
    \end{equation*}
    Fix $0<\e\ll 1$. The exponential decay \eqref{exp decay} and the similar estimate for $|\na u|$ (see Remark \ref{rmk: exp decay for der}) yield the existence of a constant $K(\e)>0$ such that, uniformly for all sufficiently large $R$,
    \begin{equation}\label{I2 est 1}
        \int_{|s|>K(\e)} |s|\left(\f12|\na u(R\bt_{ij}+s\bn_{ij})|^2+W(u(R\bt_{ij}+s\bn_{ij}))\right)\,ds < \frac{\e}{2}. 
    \end{equation}
    By increasing $K$ if necessary and using the exponential convergence of the heteroclinic connection at infinity, we may also ensure that 
    \begin{equation}\label{I2 est 2}
        \int_{|s|>K(\e)} |s|\left(\f12|U_{ij}'(s-h_{ij})|^2+W(U_{ij}(s-h_{ij}))\right)\,ds < \frac{\e}{2}.
    \end{equation}
    
    Keeping $K(\e)$ fixed, \eqref{1d asymptotic} gives 
    \begin{equation}\label{I2 est 3}
    \begin{split}
        &\lim\limits_{R\to\infty}\int_{-K(\e)}^{K(\e)}  s\left( \f12|\pa_{\bn_{ij}} u(R\bt_{ij}+s\bn_{ij})|^2 +W(u(R\bt_{ij}+s\bn_{ij})) \right)\,ds\\
         &\qquad =\int_{-K(\e)}^{K(\e)} s\left(\f12|U_{ij}'(s-h_{ij})|^2+W(U_{ij}(s-h_{ij}))\right)\,ds.
    \end{split}
    \end{equation}

    Moreover, we invoke the following estimate from \cite[Lemma 3.8]{geng2025rigidity}:
    \begin{equation*}
        \int_{\{\bx\cdot \bt_{ij}>0\}} |\pa_{\bt_{ij}} u|^2 \,d\bx <\infty.
    \end{equation*}
    This estimate, together with the uniform bound on $|D^2 u|$, implies that $|\pa_{\bt_{ij}} u(R\bt_{ij}+s\bn_{ij})|$ converges to $0$ uniformly on $s\in[-K(\e),K(\e)]$ as $R$ tends to $\infty$. Combining this with \eqref{I2 est 1},  \eqref{I2 est 2} and \eqref{I2 est 3} implies the existence of $R(\e)$ such that for all $R>R(\e)$, 
    \begin{equation*}
        \bigg| I^2_{ij}(R)+ \int_{-\infty}^{\infty} s\left(\f12|U_{ij}'(s-h_{ij})|^2+W(U_{ij}(s-h_{ij}))\right)\,ds \bigg| <\e.
    \end{equation*}
    Since $\e$ is arbitrarily small, we arrive at
    \begin{equation*}
        \lim\limits_{R\to\infty} I^2_{ij}(R) =-\int_{-\infty}^{\infty} s\left(\f12|U_{ij}'(s-h_{ij})|^2+W(U_{ij}(s-h_{ij}))\right)\,ds .
    \end{equation*}
    Combining with \eqref{tensor comp}, we obtain
    \begin{equation*}
        \sum\limits_{(i,j)} \int_{-\infty}^{\infty} s\left(\f12|U_{ij}(s-h_{ij})'|^2+W(U_{ij}(s-h_{ij}))\right)\,ds =0,
    \end{equation*}
    which by symmetry directly implies
    \begin{equation*}
        (h_{12}+h_{23}+h_{31})\int_\mathbb{R} \left(\f12|U_{ij}'|^2+W(U_{ij})\right)\,ds=0 \Rightarrow h_{12}+h_{23}+h_{31}=0.
    \end{equation*}
    The proof is complete.
    
    \end{proof}

\begin{rmk}\label{rmk: z}
    Since $h_{12}+h_{23}+h_{31}=0$, there exists a unique $\mathbf{z}\in \BR^2$ such that 
    \begin{equation*}
        h_{12}=\bn_{12}\cdot \mathbf{z},\quad h_{23}=\bn_{23}\cdot \mathbf{z},\quad h_{31}=\bn_{31}\cdot \mathbf{z}.
    \end{equation*}
    Moreover, it is easy to verify from Proposition \ref{prop: u* sol} and the definition of $u_*$ \eqref{map:u*} that $u_*(\bx-\mathbf{z})$ satisfies \eqref{1d symmetry one leg} with the same $h_{ij}$. 
\end{rmk}

For $u=(u_1,u_2)$, we define the $2\times 2$ matrix field
\begin{equation}
    \label{def of N} \mathcal{N}(u):=\sqrt{2} \begin{pmatrix}
        \frac{1+u_1}{2}-u_1^2 & -\frac{u_2}{2}-u_1u_2\\
        -\frac{u_2}{2}-u_1u_2 & \frac{1-u_1}{2}-u_2^2
    \end{pmatrix}.
\end{equation}
In \eqref{nabla u} we showed that the explicit solution $u_*$ satisfies the first-order system $\nabla u_*=\mathcal{N}(u_*)$. We now extend this relation to all minimizing triple-junction solutions $u$. 

\begin{proposition}\label{prop: 1st order eq}
    Let $u$ be an entire minimizing solution of \eqref{EL eq} satisfying \eqref{L1 conv triple} and \eqref{1d symmetry one leg}. Then $u$ solves
    \begin{equation}\label{1st order eq}
        \nabla u(\bx)= \mathcal{N}(u(\bx))\quad \text{in } \mathbb{R}^2.
    \end{equation}
\end{proposition}

\begin{proof}
    By \eqref{def of N}, direct computation gives
    \begin{equation*}
    \f12|\mathcal{N}(u)|^2= |u|^4+3u_1u_2^2-u_1^3-\f12|u|^2+\f12,
    \end{equation*}
    \begin{equation*}
        \mathrm{det}\mathcal{N}(u)= u_1^3-3u_1u_2^2-\f32|u|^2+\f12.
    \end{equation*}
    Recalling formula \eqref{potential specific}, we obtain the identity 
    \begin{equation}\label{W rep by N}
        W(u)=\f12|\mathcal{N}(u)|^2 +\f13 \mathrm{det}\mathcal{N}(u).
    \end{equation}
    We further compute
    \begin{equation}\label{calibration 1}
    \begin{split}
        &\f12|\na u|^2 +W(u)- \f12|\na u-\mathcal{N}(u)|^2 \\
        =&\f13\mathrm{det}(\mathcal{N}(u)) +\mathcal{N}(u):\nabla u\\
        =& \f13 \mathrm{det}(\na u-\mathcal{N}(u))-\f13 \mathrm{det}(\na u) + \left(\f13 \mathrm{Cof}(\mathcal{N}(u))+\mathcal{N}(u)\right): \na u,
    \end{split}
    \end{equation}
    where for a $2\times 2$ matrix $M=\begin{pmatrix}
            a & b\\
            c & d
        \end{pmatrix}$, $\mathrm{Cof}(M)$ is defined by $\begin{pmatrix}
            d & -c\\
            -b & a
        \end{pmatrix}$. 
  
By the divergence theorem, $\int_\Omega \mathrm{det}(\na u)\,d\bx=\int_{\pa \Omega} u_1\pa_\tau u_2\,d\mathcal{H}^1$, so $-\f13\mathrm{det}(\na u)$ is a null-Lagrangian. In addition, a further computation gives 
\begin{equation*}
  \f13 \mathrm{Cof}(\mathcal{N}(u))+\mathcal{N}(u)=\sqrt2\begin{pmatrix}
      \f23+\f{u_1}{3}-u_1^2-\f{u_2^2}{3} & -\f{u_2}{3}-\f23u_1u_2\\
      -\f{u_2}{3}-\f23u_1u_2 & \f23-\f{u_1}{3}-u_2^2-\f{u_1^2}{3} 
  \end{pmatrix} =D\Phi(u),
\end{equation*}
where
\begin{equation*}
    \Phi(u):= \sqrt2\begin{pmatrix}
        -\frac{|u^2|u_1}{3}-\f{u_2^2-u_1^2}{6}+\frac{2u_1}{3} \\
         -\frac{|u^2|u_2}{3}-\f{u_1u_2}{3}+\frac{2u_2}{3} .
    \end{pmatrix}
\end{equation*}
Therefore, $\left(\f13 \mathrm{Cof}(\mathcal{N}(u))+\mathcal{N}(u)\right): \na u= \mathrm{div}(\Phi(u))$ is also a null-Lagrangian. Then we rearrange \eqref{calibration 1} to obtain
\begin{equation}\label{calibration 2}
\begin{split}
    &\f12|\na u|^2 +W(u)+\underbrace{\left(\f13 \mathrm{det}(\na u)- \mathrm{div}(\Phi(u))\right)}_{\text{null-Lagrangian}}\\
    &=\f12|\na u-\mathcal{N}(u)|^2+\f13 \mathrm{det}(\na u-\mathcal{N}(u))\geq \f13|\na u-\mathcal{N}(u)|^2.
\end{split}
\end{equation}

Take $R\gg1$. Recall the following notations in the proof of Lemma \ref{lem:h sum 0}. 
\begin{equation*}
    \mathcal{T}_R=\bigcap\limits_{(i,j)\in\{(1,2),(2,3),(3,1)\}}\{\bx\cdot \bt_{ij}< R\},\quad l_{ij,R}=\{R\bt_{ij}+s\bn_{ij}: |s|<\sqrt3 R\}.
\end{equation*}
Let $\mathbf{z}\in \BR^2$ be the vector determined in Remark \ref{rmk: z} such that $\mathbf{z}\cdot\bn_{ij}=h_{ij}$. For convenience, we write
\begin{equation*}
    u_{\mathbf{z}}^*(\bx)=u_*(\bx-\mathbf{z}).
\end{equation*}

The asymptotic convergence \eqref{1d symmetry one leg}, together with the exponential decay away from the interfaces, implies that $u(\bx)$ and $ u_{\mathbf{z}}^*(\bx)$ have almost identical traces on each side of $\mathcal{T}_R$. More precisely, we have 
\begin{equation}\label{u close to u_* on paT}
    \lim\limits_{R\ri\infty} \left(\|u- u_{\mathbf{z}}^*\|_{C^2(l_{ij,R})}+\|u- u_{\mathbf{z}}^*\|_{H^1(l_{ij,R})}\right) =0.
\end{equation}
We point out that the $C^2$ convergences follows directly from \eqref{1d symmetry one leg}, while the $H^1$ convergence is a consequence of the $C^2$ convergence plus uniform exponential decays away from the interface. 

We now compare the energies $J(u,\mathcal{T}_R)$ and $J( u_{\mathbf{z}}^*,\mathcal{T}_R)$. Define an energy competitor $v$ on $\mathcal{T}_R$ by
\begin{align*}
    &\qquad\qquad \qquad v(\bx)= u_{\mathbf{z}}^*(\bx), \quad \bx\in \mathcal{T}_{R-1},\\
    &v(\bx)=  u_{\mathbf{z}}^*(\bx)+ \eta(r)\left( u(\bx\frac{R}{r})-u_{\mathbf{z}}^*(\bx\frac{R}{r})  \right),\quad \bx\in\pa \mathcal{T}_r,\ r\in (R-1,R]. 
\end{align*}
Here we choose $\eta\in C^\infty([R-1,R])$ as a smooth cutoff function such that $\eta(R-1)=0,\ \eta(R)=1,\ |\eta'|\leq 2$. Then $v=u$ on $\pa \mathcal{T}_R$. On the boundary layer $A_R:=\mathcal{T}_R\setminus \mathcal{T}_{R-1}$, by \eqref{u close to u_* on paT} we have 
\begin{align*}
    &J(v,A_R)-J( u_{\mathbf{z}}^*,A_R)\\
    &\qquad\leq \|\na u_{\mathbf{z}}^*\|_{L^2(A_R)}\|\na(v-u_{\mathbf{z}}^*)\|_{L^2(A_R)}+\f12\|\na(v-u_{\mathbf{z}}^*)\|_{L^2(A_R)}^2\\
    &\qquad\quad +\|W_u(u_\mathbf{z}^*)\|_{L^2(A_R)}\|v-u_{\mathbf{z}}^*\|_{L^2(A_R)}+C\|v-u_{\mathbf{z}}^*\|_{L^2(A_R)}^2   \\
    &\qquad\leq  C \left( \sum_{(i,j)}\bigg( \|u- u_{\mathbf{z}}^*\|_{H^1(l_{ij,R})}+\|u- u_{\mathbf{z}}^*\|_{H^1(l_{ij,R})}^2  \bigg)\right)=o(1),\quad R\to\infty.
\end{align*}

By minimality of $u$ we infer that 
\begin{equation}\label{comp ene u u_*}
    J(u,\mathcal{T}_R)\leq J(v,\mathcal{T}_R)\leq J( u_{\mathbf{z}}^*,\mathcal{T}_R)+o(1),\quad \text{as }R\to\infty.
\end{equation}

\eqref{u close to u_* on paT} also implies that the null-Lagrangian terms for $u$ and $u_{\mathbf{z}}^*$ agree asymptotically:
\begin{equation}\label{non-lag}
\begin{split}
    &\lim\limits_{R\to\infty} \int_{\mathcal{T}_R} \left\{\left(\f13 \mathrm{det}(\na u)- \mathrm{div}(\Phi(u))\right) -\left(\f13 \mathrm{det}(\na u_{\mathbf{z}}^*)- \mathrm{div}(\Phi(u_{\mathbf{z}}^*))\right)\right\}\,d\bx\\
    =& \lim\limits_{R\to\infty} \int_{\pa \mathcal{T}_R} \left\{\left(\f13u_1\pa_{\bt} u_2 - \Phi(u)\cdot \nu \right)    -\left(\f13u_{\mathbf{z},1}^*\pa_{\bt} u_{\mathbf{z},2}^*- \Phi(u_{\mathbf{z}}^*)\cdot \nu \right)\right\}\,d\mathcal{H}^1 =0.
\end{split} 
\end{equation}

Combining \eqref{calibration 2}, \eqref{comp ene u u_*} and \eqref{non-lag}, we obtain
\begin{align*}
    \int_{\mathcal{T}_R} \f13&\left(|\nabla u-\mathcal{N}(u)|^2\right)\,d\bx\leq J(u,\mathcal{T}_R) +\int_{\mathcal{T}_R} \left(\f13 \mathrm{det}(\na u)- \mathrm{div}(\Phi(u))\right) \,d\bx \\
    &\leq \int_{\mathcal{T}_R} \left( \f12|\na  u_{\mathbf{z}}^*|^2+W(u_{\mathbf{z}}^*) +\f13 \mathrm{det}(\na u_{\mathbf{z}}^*)-\mathrm{div} (\Phi(u_{\mathbf{z}}^*))\right)\,d\bx+o(1)\\
    &= \int_{\mathcal{T}_R} \left( \f12|\na u_{\mathbf{z}}^*-\mathcal{N}(u_{\mathbf{z}}^*)|^2+ \f13\mathrm{det}(\na u_{\mathbf{z}}^*-\mathcal{N}(u_{\mathbf{z}}^*)) \right)\,d\bx+o(1)\\
    &=o(1),\quad \text{as } R\to\infty,
\end{align*}
where we have used the relation $\na u_*=\mathcal{N}(u_*)$. Finally, taking $R\to\infty$ yields $\na u=\mathcal{N}(u)$ everywhere in $\BR^2$. The proof is complete.

\end{proof}

\begin{rmk}\label{rmk:u* min}
    Proposition \ref{prop: 1st order eq} and its proof also show that $u_*$ is an entire minimizing solution. Indeed, let $K\subset\mathbb{R}^2$ be a bounded open set and let $v$ be a competitor such that $v-u_*\in H_0^1(K;\mathbb{R}^2)$. Since the null-Lagrangian terms in \eqref{calibration 2} depend only on the boundary trace, they have the same integrals for $v$ and $u_*$. Moreover, $\na u_*-\mathcal{N}(u_*)=0$ and \eqref{calibration 2} together imply $J(v,K)\geq J(u_*,K)$, thus proving the minimality of $u_*$.
\end{rmk}

The next lemma shows that $u$ only takes values in the interior of the convex hull of $\{a_1,a_2,a_3\}$. 

\begin{lemma}\label{lem: convex hull}
    There exist $\lambda_1(\bx),\,\lambda_2(\bx),\,\lambda_3(\bx)\in C^\infty(\BR^2; (0,1))$ such that 
    \begin{equation}\label{lambda repre}
        u(\bx)=\sum_{i=1}^3 \lambda_i(\bx) a_i,\quad \sum_{i=1}^3 \lambda_i(\bx)=1.
    \end{equation}
\end{lemma}
    
\begin{proof}
    Set
    \begin{equation*}
        \lambda_i(\bx):=\frac{1+2a_i\cdot u(\bx)}{3},\quad i\in\{1,2,3\}.
    \end{equation*}
    Since $a_i\cdot a_i=1$, $a_i\cdot a_j=-\f12$ for $i\neq j$ and $\sum_{i=1}^3 a_i=0$, we can easily verify that $\{\lambda_i(\bx)\}_{i=1}^3$ satisfies \eqref{lambda repre} and they are uniquely determined. It remains to show $\lambda_i(\bx)>0$. 

    By \eqref{ai otimes} and \eqref{def of N} we have
    \begin{equation*}
    \mathcal{N}(u)=\sqrt{2}(\sum_i \lambda_i(\bx) a_i\otimes a_i-u\otimes u).    
    \end{equation*}
    Direct computation shows
    \begin{align*}
        \na \lambda_i&= \f23 a_i\cdot \na u\\
   (\text{since }\mathcal{N}(u)\text{ is symmetric})     &=\frac{2\sqrt2}{3}(\sum_j \lambda_j a_j\otimes a_j-u\otimes u) a_i\\
        &=\frac{2\sqrt2}{3}(\lambda_ia_i-\f12\sum_{j\neq i} \lambda_ja_j-\f{3\lambda_i-1}{2}u)\\
        &=\sqrt2 \lambda_i(a_i-u).
    \end{align*}
    Since $a_i-u(\bx)$ is smooth, we conclude that either $\lambda_i(\bx)\equiv 0$ or $\lambda_i(\bx)$ does not change sign for all $\bx\in\mathbb{R}^2$. Using the fact $\lambda_i(\bx)\ri 1$ as $|\bx|\to\infty$ along any fixed ray compactly contained in $\Om_i$, we conclude that $\lambda_i(\bx)>0$ on $\mathbb{R}^2$ for each $i\in\{1,2,3\}$. This proves the lemma.
    
\end{proof}

We are now ready to identify $u(\bx)$ as $u_*(\bx-\mathbf{z})$. 

\begin{proposition}\label{prop: 3 junciton}
    Let $u$ be an entire minimizing solution of \eqref{EL eq} satisfying
    \begin{equation*}
    u(r\bx)\xrightarrow[r\to\infty]{L^1_{loc}} \sum_{i=1}^3 a_i\mathbf{1}_{\Om_i},\quad \Om_i\text{ given by \eqref{def Omi}}.
\end{equation*}  
Then there exists $\mathbf{z}\in \mathbb{R}^2$ such that 
\begin{equation*}
    u(\bx)\equiv u_*(\bx-\mathbf{z}),
\end{equation*}
where $u_*$ is the map defined in \eqref{map:u*}.
\end{proposition}

\begin{proof}
By Lemma \ref{lem: convex hull}, we consider $\lambda_i(\bx)=\frac{1+2a_i\cdot u(\bx)}{3}>0$ for $i\in\{1,2,3\}$ which satisfies
\begin{equation*}
    \nabla \lambda_i =\sqrt{2}\lambda_i (a_i-u).
\end{equation*}
Then for any $i,j\in\{1,2,3\}$,
\begin{equation*}
    \na \log \f{\lambda_i}{\lambda_j} =\sqrt{2} (a_i-a_j)\ \Rightarrow\  \frac{\lambda_i(\bx)}{\lambda_j(\bx)}=\frac{\lambda_i(0)}{\lambda_j(0)} e^{\sqrt2 (a_i-a_j)\cdot \bx}.
\end{equation*}
Fix $j=1$. $\sum_{i=1}^3 \lambda_i(\bx)=1$ implies 
\begin{align*}
   &\qquad\qquad \lambda_1(x)\left( \sum_{i=1}^3 \f{\lambda_i(0)}{\lambda_1(0)} e^{\sqrt2 (a_i-a_1)\cdot\bx}\right)=1,\\
   &\Rightarrow \qquad \ \lambda_1(\bx) =\f{\lambda_1(0)e^{\sqrt2 a_1\cdot \bx}}{\sum_{i=1}^3 \lambda_i(0)e^{\sqrt2 a_i\cdot\bx}},\ \ \forall \bx\in\mathbb{R}^2.
\end{align*}
Similarly, for each $j\in \{1,2,3\}$,
\begin{equation}\label{formula lambda_j}
    \lambda_j(\bx) =\f{\lambda_j(0)e^{\sqrt2 a_j\cdot \bx}}{\sum_{i=1}^3 \lambda_i(0)e^{\sqrt2 a_i\cdot\bx}}.
\end{equation}
Set
\begin{equation*}
    \Lambda:= \f13 \sum_{i=1}^3 \log \lambda_i(0).
\end{equation*}
Since 
$$
\sum_{i=1}^3 (\log\lambda_i(0)-\Lambda)=0,\quad \sum_{i=1}^3a_i=0,
$$
there exists a unique point $\mathbf{z}\in \mathbb{R}^2$ satisfying
\begin{equation*}
    \sqrt2\,a_i\cdot \mathbf{z}=\Lambda-\log \lambda_i(0),\quad i\in\{1,2,3\}.
\end{equation*}
Substituting this into the above formula for $\lambda_j$ gives 
\begin{equation*}
    \lambda_j(\bx)= \f{e^{\sqrt2 a_j\cdot (\bx-\mathbf{z})}}{\sum_{i=1}^3e^{\sqrt2 a_i\cdot (\bx-\mathbf{z})}}.
\end{equation*}
It is obvious that $\mathbf{z}$ defined here has the same $\bn_{ij}$ projections and hence equals the $\mathbf{z}$ from Remark \ref{rmk: z}. This completes the proof. 

\end{proof}

The preceding arguments also apply to solutions that minimize only within the $D_3$-equivariant class.
\begin{corol}
    Let $u$ be an entire solution of \eqref{EL eq} and satisfy the following conditions
    \begin{itemize}
        \item $u$ is $D_3$-equivariant, i.e. $u\circ g=g\circ u$ for $g\in D_3$.
        \item It has the triple-junction asymptotics \eqref{pre:L1 conv} \eqref{1d asymptotic} at infinity  with $\Om_i$ given by \eqref{def Omi} and $h_{ij}=0$ due to $D_3$-symmetry. 
        \item For any $v\in H_0^1(K;\BR^2)$ with $K\Subset\BR^2$ bounded, open and $D_3$-invariant, such that $u+v$ is also $D_3$-equivariant, one has $J(u,K)\leq J(u+v,K)$.  
    \end{itemize} 
    Then,
    \begin{equation*}
        u(\bx)\equiv u_*(\bx).
    \end{equation*}
\end{corol}
\begin{proof}
The existence of entire $D_3$ equivariant solutions that satisfy these conditions has been established in \cite{bronsard1996three,alikakos2011entire}. The proof of the classification is almost identical to that of Proposition \ref{prop: 1st order eq} and Proposition \ref{prop: 3 junciton}. The only point that requires care is that the competitor $v$ used to compare $J(u,\mathcal{T}_R)$ and $J(u_*,\mathcal{T}_R)$ must remain $D_3$-equivariant,  which automatically holds due to the symmetry of $u$, $u_*$ and $\mathcal{T}_R$. Finally, solving the first-order system $\na u=\mathcal{N}(u)$ gives \eqref{formula lambda_j}, and $D_3$-symmetry forces $\lambda_1(0)=\lambda_2(0)=\lambda_3(0)=\f13$. Therefore $u\equiv u_*$.

\end{proof}

We have therefore identified $u_*$, up to translation and orthogonal change of coordinates, as the unique entire minimizing solution with triple-junction asymptotics. It is also the unique solution minimizing within the $D_3$-equivariant class. We now have all ingredients for the proof of the main theorem.
\begin{proof}[Proof of Theorem \ref{main thm}]
Theorem \ref{main thm} directly follows from Proposition \ref{prop: 1 phase}, Proposition \ref{prop: 2 phase} and Proposition \ref{prop: 3 junciton}, which classify the one-phase, two-phase and triple-junction solutions, respectively. 
\end{proof}

\bibliographystyle{acm}
\bibliography{bib-cla}

\end{document}